\documentclass[10pt,a4paper]{article}
\usepackage{geometry}
\usepackage{indentfirst}
\usepackage{amsfonts}
\usepackage{amssymb}
\usepackage{mathrsfs}
\usepackage{amsmath}
\usepackage{amsthm}
\usepackage{enumerate}
\usepackage{cite}
\allowdisplaybreaks
\newtheorem{defn}{Definition}[section]
\newtheorem{thm}[defn]{Theorem}
\newtheorem{lem}[defn]{Lemma}
\newtheorem{prop}[defn]{Proposition}
\newtheorem{cor}[defn]{Corollary}
\newtheorem{ex}[defn]{Example}
\newtheorem{re}[defn]{Remark}
\begin{document}
\title{{\bf Structure of simple multiplicative Hom-Jordan algebras}}
\author{\normalsize \bf Chenrui Yao,  Yao Ma,  Liangyun Chen}
\date{{\small{ School of Mathematics and Statistics,  Northeast Normal University,\\ Changchun 130024, CHINA}}} \maketitle
\date{}

   {\bf\begin{center}{Abstract}\end{center}}

In this paper, we mainly study structure of simple multiplicative Hom-Jordan algebras. We discuss equivalent conditions for multiplicative Hom-Jordan algebras being solvable, simple and semi-simple. Applying these theorems, we give a theorem on the classification of simple multiplicative Hom-Jordan algebras. Moreover, some propositions about bimodules of multiplicative Hom-Jordan algebras are also obtained.

\textbf{2010 Mathematics Subject Classification:} 17C10, 17C20, 17C50, 17B10, 17A30
\renewcommand{\thefootnote}{\fnsymbol{footnote}}
\footnote[0]{Keywords: \,  Bimodules, simple, classification, structure, Hom-Jordan algebras}
\footnote[0]{ Corresponding author(L. Chen): chenly640@nenu.edu.cn.}
\footnote[0]{Supported by  NNSF of China (Nos. 11771069 and 11801066), NSF of  Jilin province (No. 20170101048JC),   the project of Jilin province department of education (No. JJKH20180005K) and the Fundamental Research Funds for the Central Universities(No. 130014801).}

\section{Introduction}
Algebras where the identities defining the structure are twisted by a homomorphism are called Hom-algebras. Many predecessors have investigated Hom-algebras in the literature recently. The theory of Hom-algebras started from Hom-Lie algebras introduced and discussed in \cite{J1, D1, D2, D3}. Hom-associative algebras were introduced in \cite{M3} while Hom-Jordan algebras were introduced in \cite{M2} as twisted generalization of Jordan algebras.

In recent years, the vertex operator algebras are becoming more and more popular because of its importance. As a result, more and more people study the work about the vertex operator algebras. In \cite{L2}, by using the structure of Heisenberg algebras, Lam constructed a vertex operator algebra such that the weight  two space $V_{2} \cong A$ for a given simple Jordan algebra $J$ of type A, B or C over $\mathbb{C}$. In \cite{AT},  Ashihara gave a counterexample to the following assertion:  If $R$ is a subalgebra of the Griess algebra, then the weight two space of the vertex operator subalgebra VOA$(R)$ generated by $R$ coincides with $R$ by using a vertex operator algebra associated with the simple Jordan algebra of type D. Zhao, H. B. constructed simple quotients $\bar{V}_{\mathscr{J}, r}$ for $r \in \mathbb{Z}_{\neq 0}$ using dual-pair type constructions, where $\bar{V}_{\mathscr{J}, r}$ satisfies that $(\bar{V}_{\mathscr{J}, r})_{0} = \mathbb{C}1$, $(\bar{V}_{\mathscr{J}, r})_{1} = \{0\}$ and $(\bar{V}_{\mathscr{J}, r})_{2}$ is isomorphic to the type B Jordan algebra $\mathscr{J}$. Moreover, he reproved that $V_{\mathscr{J}, r}$ is simple if $r \notin \mathbb{Z}$ in his paper \cite{ZHB}.

The structure of Hom-algebras seems to be more complex because of variety of twisted maps. But structure of original algebras are pretty clear. So one of the ways to study the structure of Hom-algebras is to look for relationships between Hom-algebras and their induced algebras. In \cite{M3}, Makhlouf and Silvestrov introduced structure of Hom-associative algebras and Hom-Leibniz algebras together with their induced algebras. In \cite{L1}, Li, X. X. studied the structure of multiplicative Hom-Lie algebras and gave equivalent conditions for a multiplicative Hom-Lie algebra to be solvable, simple and semi-simple. By a similar analysis, we generalize above results in Hom-Jordan algebras successfully in this paper.

It's well known that simple algebras play an important role in structure theory. Similarly, it is very necessary to study simple Hom-algebras in Hom-algebras theory. In \cite{C1}, Chen, X. and Han, W. gave a classification theorem about simple multiplicative Hom-Lie algebras. Using some theorems got in Section \ref{se:3}, we generalize the above theorem into Hom-Jordan algebras.

Nowadays, one of the most modern trends in mathematics has to do with representations and deformations. The two topics are important tools in most parts of Mathematics and Physics. Representations of Hom-Lie algebras were introduced and studied in \cite{B2, S1}. But representations of Hom-Jordan algebras came much later. In 2018, Attan gave the definition of bimodules of Hom-Jordan algebras in his paper \cite{A1}. In \cite{A2}, Agrebaoui, Benali and Makhlouf study representations of simple Hom-Lie algebras and gave some propositions about it. In this paper, we also give some propositions about bimodules of Hom-Jordan algebras using similar methods.

The paper is organised as follows: In Section \ref{se:2}, we will introduce some basic definitions and prove a few lemmas which can be used in what follows. In Section \ref{se:3}, we mainly show that three important theorems, Theorem \ref{thm:3.3}, \ref{thm:3.5} and \ref{thm:3.6}, which are about solvability, simpleness and semi-simpleness of multiplicative Hom-Jordan algebras respectively. In Section \ref{se:4}, we give a theorem(Theorem \ref{thm:4.1}) on the construction of $n$-dimensional simple Hom-Jordan algebras at first. Next we give our main theorem in this section, Theorem \ref{thm:4.3}, which is about classification of simple multiplicative Hom-Jordan algebras. In Section \ref{se:5}, we'll prove a very important theorem, Theorem \ref{thm:5.5}, which is about relationships between bimodules of Hom-Jordan algebras and modules of their induced Jordan algebras. Moreover, some propositions about bimodules of simple Hom-Jordan algebras are also obtained as an application of Theorem \ref{thm:5.5}.
\section{Preliminaries}\label{se:2}
\begin{defn}\cite{M1}
\it{A Jordan algebra $J$ over a field $\rm{F}$ is an algebra satisfying for any $x, y \in J$,
\begin{enumerate}[(1)]
\item $x \circ y = y \circ x$;
\item $(x^{2} \circ y) \circ x = x^{2} \circ (y \circ x)$.
\end{enumerate}}
\end{defn}
\begin{defn}\cite{M2}
\it{A Hom-Jordan algebra over a field $\rm{F}$ is a triple $(V, \mu, \alpha)$ consisting of a linear space $V$, a bilinear map $\mu : V \times V \rightarrow V$ which is commutative and a linear map $\alpha : V \rightarrow V$ satisfying for any $x, y \in V$,
\[\mu(\alpha^{2}(x), \mu(y, \mu(x, x))) = \mu(\mu(\alpha(x), y), \alpha(\mu(x, x))),\]
where $\alpha^{2} = \alpha \circ \alpha$.}
\end{defn}
\begin{defn}
\it{A Hom-Jordan algebra $(V, \mu, \alpha)$ is called multiplicative if for any $x, y \in V$, $\alpha(\mu(x, y)) = \mu(\alpha(x), \alpha(y))$.}
\end{defn}
\begin{defn}\cite{B1}
\it{A subspace $W \subseteq V$ is a Hom-subalgebra of $(V, \mu, \alpha)$ if $\alpha(W) \subseteq W$ and
\[\mu(x, y) \in W,\quad\forall x, y \in W.\]}
\end{defn}
\begin{defn}\cite{B1}
\it{A subspace $W \subseteq V$ is a Hom-ideal of $(V, \mu, \alpha)$ if $\alpha(W) \subseteq W$ and
\[\mu(x, y) \in W,\quad\forall x \in W, y \in V.\]}
\end{defn}
\begin{defn}\cite{B1}
\it{Let $(V, \mu, \alpha)$ and $(V^{'}, \mu^{'}, \beta)$ be two Hom-Jordan algebras. A linear map $\phi : V \rightarrow V^{'}$ is said to be a homomorphism of Hom-Jordan algebras if
\begin{enumerate}[(1)]
\item $\phi(\mu(x, y)) = \mu^{'}(\phi(x), \phi(y))$;
\item $\phi \circ \alpha = \beta \circ \phi$.
\end{enumerate}
In particular, $\phi$ is an isomorphism if $\phi$ is bijective.}
\end{defn}
\begin{defn}
\it{A Hom-Jordan algebra $(V, \mu, \alpha)$ is called a Jordan-type Hom-Jordan algebra if there exists a Jordan algebra $(V, \mu^{'})$ such that
\[\mu(x, y) = \alpha(\mu^{'}(x, y)) = \mu^{'}(\alpha(x), \alpha(y)),\quad\forall x, y \in V,\]
and $(V, \mu^{'})$ is called the induced Jordan algebra.}
\end{defn}
\begin{lem}\label{le:2.8}
\begin{enumerate}[(1)]
\item Suppose that $(V, \mu)$ is a Jordan algebra and $\alpha : V \rightarrow V$ is a homomorphism. Then $(V, \tilde{\mu}, \alpha)$ is a multiplicative Hom-Jordan algebra with $\tilde{\mu}(x, y) = \alpha(\mu(x, y)),\;\forall x, y \in V$.
\item Suppose that $(V, \mu, \alpha)$ is a multiplicative Hom-Jordan algebra and $\alpha$ is invertible. Then $(V, \mu, \alpha)$ is a Jordan-type Hom-Jordan algebra and its induced Jordan algebra is $(V, \mu^{'})$ with $\mu^{'}(x, y) = \alpha^{-1}(\mu(x, y)),\;\forall x, y \in V$.
\end{enumerate}
\end{lem}
\begin{proof}
(1). We have $\tilde{\mu}$ is commutative since $\mu$ is commutative.

For all $x, y \in V$, we have
\begin{align*}
&\tilde{\mu}(\alpha^{2}(x), \tilde{\mu}(y, \tilde{\mu}(x,x))) = \alpha(\mu(\alpha^{2}(x), \alpha(\mu(y, \alpha(\mu(x, x))))))\\
&= \mu(\alpha^{3}(x), \mu(\alpha^{2}(y), \mu(\alpha^{3}(x), \alpha^{3}(x)))) = \mu(\mu(\alpha^{3}(x), \alpha^{2}(y)), \mu(\alpha^{3}(x), \alpha^{3}(x)))\\
&= \alpha(\mu(\alpha(\mu(\alpha(x), y)), \alpha^{2}(\mu(x, x)))) = \tilde{\mu}(\tilde{\mu}(\alpha(x), y), \alpha(\tilde{\mu}(x, x))),
\end{align*}
which implies that $(V, \tilde{\mu}, \alpha)$ is a Hom-Jordan algebra.
\[\alpha(\tilde{\mu}(x, y)) = \alpha^{2}(\mu(x, y)) = \alpha(\mu(\alpha(x), \alpha(y))) = \tilde{\mu}(\alpha(x), \alpha(y)),\]
which implies that $(V, \tilde{\mu}, \alpha)$ is multiplicative. Hence, $(V, \tilde{\mu}, \alpha)$ is a multiplicative Hom-Jordan algebra.

(2). We have $\mu^{'}$ is commutative since $\mu$ is commutative.

For any $x, y \in V$, we have
\begin{align*}
&\mu^{'}(\mu^{'}(\mu^{'}(x, x), y), x) = \alpha^{-1}(\mu(\alpha^{-1}(\mu(\alpha^{-1}(\mu(x, x)), y)), x))\\
&= \alpha^{-3}(\mu(\mu(\mu(x, x), \alpha(y)), \alpha^{2}(x))) = \alpha^{-3}(\mu(\alpha(\mu(x, x)), \mu(\alpha(y), \alpha(x))))\\
&= \alpha^{-1}(\mu(\alpha^{-1}(\mu(x, x)), \alpha^{-1}(\mu(y, x)))) = \mu^{'}(\mu^{'}(x, x), \mu^{'}(y, x)),
\end{align*}
which implies that $(V, \mu^{'})$ is a Jordan algebra.

It's obvious that $\mu(x, y) = \alpha(\mu^{'}(x, y)) = \mu^{'}(\alpha(x), \alpha(y))$ for any $x, y \in V$. Hence, $(V, \mu, \alpha)$ is a Jordan-type Hom-Jordan algebra.
\end{proof}
\begin{defn}
\it{Suppose that $(V, \mu, \alpha)$ is a Hom-Jordan algebra. Define its derived sequence as follow:
\[V^{(1)} = \mu(V, V),\;V^{(2)} = \mu(V^{(1)}, V^{(1)}),\cdots,\;V^{(k)} = \mu(V^{(k - 1)}, V^{(k - 1)}),\cdots.\]
If there exists $m \in \mathbb{Z}^{+}$ such that $V^{(m)} = 0$, then $(V, \mu, \alpha)$ is called solvable.}
\end{defn}
\begin{defn}
\it{Suppose that $(V, \mu, \alpha)$ is a Hom-Jordan algebra and $\alpha \neq 0$. If $(V, \mu, \alpha)$ has no non trivial Hom-ideals and satisfies $\mu(V, V) = V$, then $(V, \mu, \alpha)$ is called simple. If
\[V = V_{1} \oplus V_{2} \oplus \cdots \oplus V_{s},\]
where $V_{i}(1 \leq i \leq s)$ are simple Hom-ideals of $(V, \mu, \alpha)$, then $(V, \mu, \alpha)$ is called semi-simple.}
\end{defn}
\begin{prop}\label{prop:2.11}
Suppose that $(V_{1}, \tilde{\mu_{1}}, \alpha)$ and $(V_{2}, \tilde{\mu_{2}}, \beta)$ are two Jordan-type Hom-Jordan algebras and $\beta$ is injective. Then $\phi$ is an isomorphism from $(V_{1}, \tilde{\mu_{1}}, \alpha)$ to $(V_{2}, \tilde{\mu_{2}}, \beta)$ if and only if $\phi$ is an isomorphism between their induced Jordan algebras $(V_{1}, \mu_{1})$ and $(V_{2}, \mu_{2})$ and $\phi$ satisfies $\beta \circ \phi = \phi \circ \alpha$.
\end{prop}
\begin{proof}
$(\Rightarrow)$ For any $x, y \in V_{1}$, we have
\[\phi(\tilde{\mu_{1}}(x, y)) = \tilde{\mu_{2}}(\phi(x), \phi(y)),\]
i.e.,
\[\phi(\alpha(\mu_{1}(x, y))) = \beta(\mu_{2}(\phi(x), \phi(y))).\]
Note that $\phi \circ \alpha = \beta \circ \phi$, we have
\[\beta(\phi(\mu_{1}(x, y))) = \beta(\mu_{2}(\phi(x), \phi(y))).\]
Since $\beta$ is injective, we have
\[\phi(\mu_{1}(x, y)) = \mu_{2}(\phi(x), \phi(y)),\]
which implies that $\phi$ is an isomorphism from $(V_{1}, \mu_{1})$ to $(V_{2}, \mu_{2})$.

$(\Leftarrow)$ For any $x, y \in V_{1}$, we have
\begin{align*}
&\phi(\tilde{\mu_{1}}(x, y)) = \phi(\alpha(\mu_{1}(x, y))) = \beta(\phi(\mu_{1}(x, y)))\\
&= \beta(\mu_{2}(\phi(x), \phi(y))) = \tilde{\mu_{2}}(\phi(x), \phi(y)),
\end{align*}
note that $\beta \circ \phi = \phi \circ \alpha$, we have $\phi$ is an isomorphism from $(V_{1}, \tilde{\mu_{1}}, \alpha)$ to $(V_{2}, \tilde{\mu_{2}}, \beta)$.
\end{proof}
\begin{lem}\label{le:2.12}
Simple multiplicative Hom-Jordan algebras with $\alpha \neq 0$ are Jordan-type Hom-Jordan algebras.
\end{lem}
\begin{proof}
Suppose that $(V, \mu, \alpha)$ is a simple multiplicative Hom-Jordan algebra. According to Lemma \ref{le:2.8} (2), we only need to show that $\alpha$ is invertible. If $\alpha$ is not invertible, then $Ker(\alpha) \neq 0$. It's obvious that $\alpha(Ker(\alpha)) \subseteq Ker(\alpha)$. For any $x \in Ker(\alpha)$, $y \in V$, we have
\[\alpha(\mu(x, y)) = \mu(\alpha(x), \alpha(y)) = \mu(0, \alpha(y)) = 0,\]
which implies that $\mu(Ker(\alpha), V) \subseteq Ker(\alpha)$. Then $Ker(\alpha)$ is a non trivial Hom-ideal of $(V, \mu, \alpha)$, contradicting with $(V, \mu, \alpha)$ is simple. Therefore, $Ker(\alpha) = 0$, i.e., $\alpha$ is invertible. Hence, $(V, \mu, \alpha)$ is a Jordan-type Hom-Jordan algebra.
\end{proof}
Now we recall a corollary about Proposition \ref{prop:2.11} using Lemma \ref{le:2.12}.
\begin{cor}\label{cor:2.13}
Two simple multiplicative Hom-Jordan algebras $(V_{1}, \tilde{\mu_{1}}, \alpha)$ and $(V_{2}, \tilde{\mu_{2}}, \beta)$ are isomorphic if and only if there exists an isomorphism $\phi$ between their induced Jordan algebras $(V_{1}, \mu_{1})$ and $(V_{2}, \mu_{2})$ and $\phi$ satisfies $\beta \circ \phi = \phi \circ \alpha$.
\end{cor}
\section{Structure of Multiplicative Hom-Jordan algebras}\label{se:3}
In this section, we discuss the sufficient and necessary conditions that multiplicative Hom-Jordan algebras are solvable, simple and semi-simple.
\begin{prop}\label{prop:3.1}
Suppose that $(V, \mu, \alpha)$ is a multiplicative Hom-Jordan algebra and $I$ is a Hom-ideal of $(V, \mu, \alpha)$. Then $(V/I, \bar{\mu}, \bar{\alpha})$ is a multiplicative Hom-Jordan algebra where $\bar{\mu}(\bar{x}, \bar{y}) = \overline{\mu(x, y)},\;\bar{\alpha}(\bar{x}) = \overline{\alpha(x)}$ for all $\bar{x}, \bar{y} \in V/I$.
\end{prop}
\begin{proof}
We have $\bar{\mu}$ is commutative since $\mu$ is commutative.

For any $\bar{x}, \bar{y} \in V/I$, we have
\begin{align*}
&\bar{\mu}(\bar{\alpha}^{2}(\bar{x}), \bar{\mu}(\bar{y}, \bar{\mu}(\bar{x}, \bar{x}))) = \overline{\mu(\alpha^{2}(x), \mu(y, \mu(x, x)))}\\
&= \overline{\mu(\mu(\alpha(x), y), \alpha(\mu(x, x)))} = \bar{\mu}(\bar{\mu}(\bar{\alpha}(\bar{x}), \bar{y}), \bar{\alpha}(\bar{\mu}(\bar{x}, \bar{x}))).
\end{align*}
Hence, $(V/I, \bar{\mu}, \bar{\alpha})$ is a Hom-Jordan algebra.
\[\bar{\alpha}(\bar{\mu}(\bar{x}, \bar{y})) = \overline{\alpha(\mu(x, y))} = \overline{\mu(\alpha(x), \alpha(y))} = \bar{\mu}(\bar{\alpha}(\bar{x}, \bar{y})),\]
which implies that $(V/I, \bar{\mu}, \bar{\alpha})$ is multiplicative.
\end{proof}
\begin{cor}\label{cor:3.2}
Suppose that $(V, \mu, \alpha)$ is a multiplicative Hom-Jordan algebra and satisfies $\alpha^{2} = \alpha$. Then $(V/Ker(\alpha), \bar{\mu}, \bar{\alpha})$ is a Jordan-type Hom-Jordan algebra.
\end{cor}
\begin{proof}
If $\alpha$ is invertible, $Ker(\alpha) = 0$. According to Lemma \ref{le:2.8} (2), the conclusion is valid. If $\alpha$ isn't invertible, according to the proof of Lemma \ref{le:2.12}, we have $Ker(\alpha)$ is a Hom-ideal of $(V, \mu, \alpha)$. Then we have $(V/Ker(\alpha), \bar{\mu}, \bar{\alpha})$ is a multiplicative Hom-Jordan algebra according to Proposition \ref{prop:3.1}. Now we show that $\bar{\alpha}$ is invertible on $V/Ker(\alpha)$.

Assume that $\bar{x} \in Ker(\bar{\alpha})$. Then we have $\overline{\alpha(x)} = \bar{\alpha}(\bar{x}) = \bar{0}$, i.e., $\alpha(x) \in Ker(\alpha)$. Note that $\alpha^{2} = \alpha$, we have
\[\alpha(x) = \alpha^{2}(x) = \alpha(\alpha(x)) = 0,\]
which implies that $x \in Ker(\alpha)$, i.e., $\bar{x} = \bar{0}$. Hence, $\bar{\alpha}$ is invertible. According to Lemma \ref{le:2.8} (2), $(V/Ker(\alpha), \bar{\mu}, \bar{\alpha})$ is a Jordan-type Hom-Jordan algebra.
\end{proof}
\begin{thm}\label{thm:3.3}
Suppose that $(V, \mu, \alpha)$ is a multiplicative Hom-Jordan algebra and $\alpha$ is invertible. Then $(V, \mu, \alpha)$ is solvable if and only if its induced Jordan algebra $(V, \mu^{'})$ is solvable.
\end{thm}
\begin{proof}
Denote derived series of $(V, \mu^{'})$ and $(V, \mu, \alpha)$ by $V^{(i)}$, $\tilde{V}^{(i)}(i = 1, 2, \cdots)$ respectively.

Suppose that $(V, \mu^{'})$ is solvable. Then there exists $m \in \mathbb{Z}^{+}$ such that $V^{(m)} = 0$.

Note that
\[\tilde{V}^{(1)} = \mu(V, V) = \alpha(\mu^{'}(V, V)) = \alpha(V^{(1)}),\]
\[\tilde{V}^{(2)} = \mu(\tilde{V}^{(1)}, \tilde{V}^{(1)}) = \mu(\alpha(V^{(1)}), \alpha(V^{(1)})) = \alpha^{2}(\mu^{'}(V^{(1)}, V^{(1)})) = \alpha^{2}(V^{(2)}),\]
we have $\tilde{V}^{(m)} = \alpha^{m}(V^{(m)})$ by induction. Hence, $\tilde{V}^{(m)} = 0$, i.e., $(V, \mu, \alpha)$ is solvable.

On the other hand, assume that $(V, \mu, \alpha)$ is solvable. Then there exists $m \in \mathbb{Z}^{+}$ such that $\tilde{V}^{(m)} = 0$. We have $\tilde{V}^{(m)} = \alpha^{m}(V^{(m)})$ by the above proof. Hence we have $V^{(m)} = 0$ since $\alpha$ is invertible. Therefore, $(V, \mu^{'})$ is solvable.
\end{proof}
\begin{lem}\label{le:3.4}
Suppose that an algebra $\mathcal{A}$ over $\rm{F}$ can be decomposed into the unique direct sum of simple ideals $\mathcal{A} = \oplus^{s}_{i = 1}\mathcal{A}_{i}$ where $\mathcal{A}_{i}$ aren't isomorphic to each other and $\alpha \in Aut(\mathcal{A})$. Then $\alpha(\mathcal{A}_{i}) = \mathcal{A}_{i}(1 \leq i \leq s)$.
\end{lem}
\begin{proof}
For any $1 \leq i \leq s$, we have
\[\alpha(\mathcal{A}_{i})\mathcal{A} = \alpha(\mathcal{A}_{i})\alpha(\mathcal{A}) = \alpha(\mathcal{A}_{i}\mathcal{A}) \subseteq \alpha(\mathcal{A}_{i})\]
since $\mathcal{A}_{i}$ are ideals of $\mathcal{A}$. Similarly, we have $\mathcal{A}\alpha(\mathcal{A}_{i}) \subseteq \alpha(\mathcal{A}_{i})$. Hence, $\alpha(\mathcal{A}_{i})$ are also ideals of $\mathcal{A}$. Moreover, $\alpha(\mathcal{A}_{i})$ are simple since $\mathcal{A}_{i}$ are simple.

Note that $\mathcal{A} = \oplus^{s}_{i = 1}\mathcal{A}_{i}$, we have
\[\mathcal{A} = \alpha(\mathcal{A}) = \alpha(\oplus^{s}_{i = 1}\mathcal{A}_{i}) = \oplus^{s}_{i = 1}\alpha(\mathcal{A}_{i}).\]

Note that the decomposition is unique, there exists $1 \leq j \leq s$ such that $\alpha(\mathcal{A}_{i}) = \mathcal{A}_{j}$ for any $1 \leq i \leq s$.

If $j \neq i$, then we have
\[\mathcal{A}_{i} \cong \alpha(\mathcal{A}_{i}) = \mathcal{A}_{j},\]
contradicting with the assumption that $\mathcal{A}_{i}$ aren't isomorphic to each other.
Hence, we have $\alpha(\mathcal{A}_{i}) = \mathcal{A}_{i}(1 \leq i \leq s)$ for any $s \in \mathbb{N}$.
\end{proof}
\begin{thm}\label{thm:3.5}
\begin{enumerate}[(1)]
\item Suppose that $(V, \mu, \alpha)$ is a simple multiplicative Hom-Jordan algebra. Then its induced Jordan algebra $(V, \mu^{'})$ is semi-simple. Moreover, $(V, \mu^{'})$ can be decomposed into direct sum of isomorphic simple ideals, in addition, $\alpha$ acts simply transitively on simple ideals of the induced Jordan algebra.
\item Suppose that $(V, \mu^{'})$ is a simple Jordan algebra and $\alpha \in Aut(V)$. Define $\mu : V \times V \rightarrow V$ by
\[\mu(x, y) = \alpha(\mu^{'}(x, y)),\quad\forall x, y \in V,\]
then $(V, \mu, \alpha)$ is a simple multiplicative Hom-Jordan algebra.
\end{enumerate}
\end{thm}
\begin{proof}
(1) According to the proof of Lemma \ref{le:2.8} (2) and Lemma \ref{le:2.12}, $\alpha$ is an automorphism both on $(V, \mu, \alpha)$ and $(V, \mu^{'})$.

Suppose that $V_{1}$ is the maximal solvable ideal of $(V, \mu^{'})$. Then there exists $m \in \mathbb{Z^{+}}$ such that $V_{1}^{(m)} = 0$.

Note that
\[\mu^{'}(\alpha(V_{1}), V) = \mu^{'}(\alpha(V_{1}), \alpha(V)) = \alpha(\mu^{'}(V_{1}, V)) \subseteq \alpha(V_{1}),\]
\[(\alpha(V_{1}))^{(m)} = \alpha(V_{1}^{(m)}) = 0,\]
we have $\alpha(V_{1})$ is also a solvable ideal of $(V, \mu^{'})$. Then we have $\alpha(V_{1}) \subseteq V_{1}$. Moreover,
\[\mu(V_{1}, V) = \alpha(\mu^{'}(V_{1}, V)) \subseteq \alpha(V_{1}) \subseteq V_{1},\]
so $V_{1}$ is a Hom-ideal of $(V, \mu, \alpha)$. Then we have $V_{1} = 0$ or $V_{1} = V$ since $(V, \mu, \alpha)$ is simple. If $V_{1} = V$, according to the proof of Theorem \ref{thm:3.3}, we have
\[\tilde{V}^{(m)} = \alpha^{m}(V^{(m)}) = \alpha^{m}(V_{1}^{(m)}) = 0,\]
on the other hand, $V = \mu(V, V)$ since $(V, \mu, \alpha)$ is simple. Then we have $\tilde{V}^{(m)} = V$. Contradiction. Hence, $V_{1} = 0$. Therefore, $(V, \mu^{'})$ is semi-simple.

Since $(V, \mu^{'})$ is semi-simple, we have $V = \oplus^s_{i = 1}V_{i}$, where $V_{i}(1 \leq i \leq s)$ are simple ideals of $(V, \mu^{'})$. Because there may be isomorphic Jordan algebras in $V_{1}, V_{2}, \cdots, V_{s}$, we rearrange the order as following
\[V = V_{11} \oplus V_{12} \oplus \cdots \oplus V_{1m_{1}} \oplus V_{21} \oplus V_{22} \oplus \cdots \oplus V_{2m_{2}} \oplus \cdots \oplus V_{t1} \oplus V_{t2} \oplus \cdots \oplus V_{tm_{t}},\]
where
\[(V_{ij}, \mu^{'}) \cong (V_{ik}, \mu^{'}),\quad 1 \leq j, k \leq m_{i}, i = 1, 2, \cdots, t.\]
According to Lemma \ref{le:3.4}, we have
\[\alpha(V_{i1} \oplus V_{i2} \oplus \cdots \oplus V_{im_{i}}) = V_{i1} \oplus V_{i2} \oplus \cdots \oplus V_{im_{i}},\]
\begin{align*}
&\mu(V_{i1} \oplus V_{i2} \oplus \cdots \oplus V_{im_{i}}, V) = \alpha(\mu^{'}(V_{i1} \oplus V_{i2} \oplus \cdots \oplus V_{im_{i}}, V))\\
&\subseteq \alpha(V_{i1} \oplus V_{i2} \oplus \cdots \oplus V_{im_{i}}) = V_{i1} \oplus V_{i2} \oplus \cdots \oplus V_{im_{i}},
\end{align*}
we have $V_{i1} \oplus V_{i2} \oplus \cdots \oplus V_{im_{i}}$ are Hom-ideals of $(V, \mu, \alpha)$. Since $(V, \mu, \alpha)$ is simple, we have $V_{i1} \oplus V_{i2} \oplus \cdots \oplus V_{im_{i}} = 0$ or $V$. So all but one $V_{i1} \oplus V_{i2} \oplus \cdots \oplus V_{im_{i}}$ must be $0$. Without loss of generality, we can assume
\[V = V_{11} \oplus V_{12} \oplus \cdots \oplus V_{1m_{1}}.\]
When $m_{1} = 1$, $(V, \mu^{'})$ is simple. When $m_{1} > 1$, if
\[\alpha(V_{1p}) = V_{1p}(1 \leq p \leq m_{1}),\]
then $V_{1p}$ is a non trivial ideal of $(V, \mu, \alpha)$, which contradicts with the fact that $(V, \mu, \alpha)$ is simple. Hence,
\[\alpha(V_{1p}) = V_{1l}(1 \leq l \neq p \leq m_{1}).\]
In addition, it is easy to show that $V_{11} \oplus \alpha(V_{11}) \oplus \cdots \oplus \alpha^{m_{1} - 1}(V_{11})$ is a Hom-ideal of $(V, \mu, \alpha)$. Therefore,
\[V = V_{11} \oplus \alpha(V_{11}) \oplus \cdots \oplus \alpha^{m_{1} - 1}(V_{11}).\]
That is $\alpha$ acts simply transitively on simple ideals of the induced Jordan algebra.

(2) Suppose that $(V, \mu^{'})$ is a simple Jordan algebra. According to Lemma \ref{le:2.8} (1), we have $(V, \mu, \alpha)$ is a multiplicative Hom-Jordan algebra. Suppose that $V_{1}$ is a non trivial Hom-ideal of $(V, \mu, \alpha)$, then we have
\[\mu^{'}(V_{1}, V) = \alpha^{-1}(\mu(V_{1}, V)) \subseteq \alpha^{-1}(V_{1}) = V_{1}.\]
So $V_{1}$ is a non trivial ideal of $(V, \mu^{'})$, contradiction. So $(V, \mu, \alpha)$ has no non trivial ideal. If $\mu(V, V) \subsetneqq V$, then
\[\mu^{'}(V, V) = \alpha^{-1}(\mu(V, V)) \subsetneqq \alpha^{-1}(V) = V,\]
contradicting with $(V, \mu^{'})$ is a simple Jordan algebra. Hence, $(V, \mu, \alpha)$ is simple.
\end{proof}
\begin{thm}\label{thm:3.6}
\begin{enumerate}[(1)]
\item Suppose that $(V, \mu, \alpha)$ is a semi-simple multiplicative Hom-Jordan algebra. Then $(V, \mu, \alpha)$ is a Jordan-type Hom-Jordan algebra and its induced Jordan algebra $(V, \mu^{'})$ is also semi-simple.
\item Suppose that $(V, \mu^{'})$ is a semi-simple Jordan algebra and has the decomposition $V = \oplus^{s}_{i = 1}V_{i}$ where $V_{i}(1 \leq i \leq s)$ are simple ideals of $(V, \mu^{'})$. $\alpha \in Aut(V)$ satisfies $\alpha(V_{i}) = V_{i}(1 \leq i \leq s)$. Then $(V, \mu, \alpha)$ is a semi-simple multiplicative Hom-Jordan algebra and has the unique decomposition.
\end{enumerate}
\end{thm}
\begin{proof}
(1) According to the assumption, $(V, \mu, \alpha)$ has the decomposition $V = \oplus^{s}_{i = 1}V_{i}$ where $V_{i}(1 \leq i \leq s)$ are simple Hom-ideals of $(V, \mu, \alpha)$. Then $(V_{i}, \mu, \alpha|_{V_{i}})(1 \leq i \leq s)$ are simple Hom-Jordan algebras. According to the proof of Lemma \ref{le:2.12}, $\alpha|_{V_{i}}(1 \leq i \leq s)$ are invertible. Therefore, $\alpha$ is invertible on $V$. According to Lemma \ref{le:2.8} (2), $(V, \mu, \alpha)$ is a Jordan-type Hom-Jordan algebra and its induced Jordan algebra is $(V, \mu^{'})$ where $\mu^{'}(x, y) = \alpha^{-1}(\mu(x, y))$ for all $x, y \in V$.

According to the proof of Theorem \ref{thm:3.5} (2), $V_{i}(i = 1, 2, \cdots, s)$ are ideals of $(V, \mu^{'})$. Moreover, $(V_{i}, \mu^{'})$ are induced Jordan algebras of simple Hom-Jordan algebras $(V_{i}, \mu, \alpha|_{V_{i}})$ respectively. According to Theorem \ref{thm:3.5} (1), $(V_{i}, \mu^{'})$ are semi-simple Jordan algebras and can be decomposed into direct sum of isomorphic simple ideals $V_{i} = V_{i1} \oplus V_{i2} \oplus \cdots \oplus V_{im_{i}}$. Therefore, $(V, \mu^{'})$ is semi-simple and has the decomposition of direct sum of simple ideals
\[V = V_{11} \oplus V_{12} \oplus \cdots \oplus V_{1m_{1}} \oplus V_{21} \oplus V_{22} \oplus \cdots \oplus V_{2m_{2}} \oplus \cdots \oplus V_{s1} \oplus V_{s2} \oplus \cdots \oplus V_{sm_{s}}.\]

(2) According to Lemma \ref{le:2.8} (1), $(V, \mu, \alpha)$ is a multiplicative Hom-Jordan algebra.

For all $1 \leq i \leq s$, we have
\[\mu(V_{i}, V) = \alpha(\mu^{'}(V_{i}, V)) \subseteq \alpha(V_{i}) = V_{i},\]
note that $\alpha(V_{i}) = V_{i}$, we have $V_{i}$ are Hom-ideals of $(V, \mu, \alpha)$.

If there exists $V_{i0} \subsetneqq V_{i}$ is a non trivial Hom-ideal of $(V_{i}, \mu, \alpha|_{V_{i}})$, then we have
\[\mu(V_{i0}, V) = \mu(V_{i0}, V_{1} \oplus V_{2} \oplus \cdots \oplus V_{s}) = \mu(V_{i0}, V_{i}) \subseteq V_{i0},\]
so $V_{i0}$ is a non trivial Hom-ideal of $(V, \mu, \alpha)$. According to the proof of Theorem \ref{thm:3.5} (2), $V_{i0}$ is also a non trivial ideal of $(V, \mu^{'})$. Hence, $V_{i0}$ is also a non trivial ideal of $(V_{i}, \mu^{'})$. Contradiction. Hence, $V_{i}(i = 1, 2, \cdots, s)$ are simple Hom-ideals of $(V, \mu, \alpha)$. Therefore, $(V, \mu, \alpha)$ is semi-simple and has the unique decomposition.
\end{proof}
\begin{prop}
Suppose that $(V, \mu, \alpha)$ is a multiplicative Hom-Jordan algebra satisfying $\alpha^{2} = \alpha$ and $\mu(Im(\alpha), V) \subseteq Im(\alpha)$. Then $(V, \mu, \alpha)$ is isomorphic to the decomposition of direct sum of Hom-Jordan algebras
\[V \cong (V/Ker(\alpha)) \oplus Ker(\alpha).\]
\end{prop}
\begin{proof}
Set $V_{1} = (V/Ker(\alpha)) \oplus Ker(\alpha)$. According to Corollary \ref{cor:3.2}, $(V/Ker(\alpha), \bar{\mu}, \bar{\alpha})$ is a Hom-Jordan algebra. It's obvious that $(Ker(\alpha), \mu, \alpha|_{Ker(\alpha)})$ is a Hom-Jordan algebra. Define $\mu_{1} : V_{1} \times V_{1} \rightarrow V_{1}$ and $\alpha_{1} : V_{1} \rightarrow V_{1}$ by
\[\mu_{1}((\bar{x}, k_{1}), (\bar{y}, k_{2})) = (\overline{\mu(x, y)}, \mu(k_{1}, k_{2})),\]
\[\alpha_{1}((\bar{x}, k_{1})) = (\overline{\alpha(x)}, 0).\]
Then $(V_{1}, \mu_{1}, \alpha_{1})$ is a Hom-Jordan algebra and $V_{1} = (V/Ker(\alpha)) \oplus Ker(\alpha)$ is the direct sum of ideals.

Now we show that $(V, \mu, \alpha) \cong (V_{1}, \mu_{1}, \alpha_{1})$. According to the assumption, we have $Im(\alpha)$ is a Hom-ideal of $(V, \mu, \alpha)$. For any $x \in Ker(\alpha) \cap Im(\alpha)$, there exists $y \in V$ such that $x = \alpha(y)$. Then we have
\[0 = \alpha(x) = \alpha^{2}(y) = \alpha(y) = x,\]
so $Ker(\alpha) \cap Im(\alpha) = \{0\}$. So for any $x \in V$, we have $x = x - \alpha(x) + \alpha(x)$ where $x - \alpha(x) \in Ker(\alpha)$ and $\alpha(x) \in Im(\alpha)$. Therefore, $V = Ker(\alpha) \oplus Im(\alpha)$.

Obviously, $(Im(\alpha), \mu, \alpha|_{Im(\alpha)})$ is a Hom-Jordan algebra.

Next we'll show that $(Im(\alpha), \mu, \alpha|_{Im(\alpha)}) \cong (V/Ker(\alpha), \bar{\mu}, \bar{\alpha})$. Define $\varphi : V/Ker(\alpha) \rightarrow Im(\alpha)$ by $\varphi(\bar{x}) = \alpha(x)$ for all $\bar{x} \in V/Ker(\alpha)$. Obviously, $\varphi$ is bijective. For all $\bar{x}, \bar{y} \in V/Ker(\alpha)$, we have
\[\varphi(\bar{\mu}(\bar{x}, \bar{y})) = \varphi(\overline{\mu(x, y)}) = \alpha(\mu(x, y)) = \mu(\alpha(x), \alpha(y)) = \mu(\varphi(\bar{x}), \varphi(\bar{y})),\]
\[\varphi(\bar{\alpha}(\bar{x})) = \varphi(\overline{\alpha(x)}) = \alpha^{2}(x) = \alpha(\varphi(\bar{x})),\]
which implies that $\varphi \circ \bar{\alpha} = \alpha \circ \varphi$. Therefore, $\varphi$ is an isomorphism, i.e., $(Im(\alpha), \mu, \alpha|_{Im(\alpha)}) \cong (V/Ker(\alpha), \bar{\mu}, \bar{\alpha})$. Therefore, $V = Ker(\alpha) \oplus Im(\alpha) \cong (V/Ker(\alpha)) \oplus Ker(\alpha)$.
\end{proof}
\section{Classification of simple multiplicative Hom-Jordan algebras}\label{se:4}
In this section, we'll give a theorem about classification on simple multiplicative Hom-Jordan algebras. At first, we give a construction of $n$-dimensional simple Hom-Jordan algebras.
\begin{thm}\label{thm:4.1}
There exist $n$-dimensional simple Hom-Jordan algebras for any $n \in \mathbb{Z}^{+}$.
\end{thm}
\begin{proof}
When $n = 1$, let $V = \mathbb{R}^{+}$ over $\mathbb{R}$, i.e, $\mu : V \times V \rightarrow V,\;\mu(a, b) = \frac{1}{2}(ab + ba)$ for all $a, b \in \mathbb{R}$. It's obvious that $\rm{dim}\it{(V)} = 1$. Take $\alpha = k\;id_{\mathbb{R}}$ for $k \in \mathbb{R}$. Then $(V, \mu, \alpha)$ is a $1$-dimensional Hom-Jordan algebra. Obviously, $(V, \mu, \alpha)$ is simple since $(V, \mu, \alpha)$ has no non trivial Hom-ideal and $\mu(V, V) = V$.

When $n = 2$, let $\{e_{0}, e_{1}\}$ be a basis of $2$-dimensional vector space over $\mathbb{C}$. Define a bilinear symmetric binary operation $\mu : V \times V \rightarrow V$:
\[\mu(e_{0}, e_{0}) = e_{0},\;\mu(e_{1}, e_{1}) = e_{1},\;\mu(e_{0}, e_{1}) = \mu(e_{1}, e_{0}) = e_{0} + e_{1}.\]
Obviously, $\mu(V, V) = V$.
Take $\alpha \in End(V)$ where
\[\alpha(e_{0}) = pe_{0}, \alpha(e_{1}) = qe_{1},\;p, q \in \mathbb{C}.\]
One can verify that $(V, \mu, \alpha)$ is a $2$-dimensional Hom-Jordan algebra. Next we'll show that $(V, \mu, \alpha)$ is simple.

Suppose that $I$ is a non trivial Hom-ideal of $(V, \mu, \alpha)$. Then there exists $0 \neq a = t_{1}e_{0} + t_{2}e_{1} \in I$, where $t_{1}, t_{2} \in \mathbb{C}$. Then we have $(t_{1} + t_{2})e_{0} + t_{2}e_{1} = \mu(a, e_{0}) \in I$, i.e, $\frac{t_{1} + t_{2}}{t_{1}} = \frac{t_{2}}{t_{2}}$; $t_{1}e_{0} + (t_{1} + t_{2})e_{1} = \mu(a, e_{1}) \in I$, i.e, $\frac{t_{2}}{t_{1} + t_{2}} = \frac{t_{1}}{t_{1}}$ since $\rm{dim}\it{(I)} = 1$. So $t_{1} + t_{2} = t_{1}$, $t_{1} + t_{2} = t_{2}$, which imply that $t_{1} = 0, t_{2} = 0$. Hence, $I = 0$, contradiction. Therefore, $(V, \mu, \alpha)$ is a $2$-dimensional simple Hom-Jordan algebra.

When $n \geq 3$, let $\{a_{\bar{i}} | i \in \mathbb{Z}_{n}\}$ be a basis of $n$-dimensional vector space $V$ over $\mathbb{C}$. Define a bilinear symmetric binary operation $\mu : V \times V \rightarrow V$:
\[\mu(a_{\bar{i}}, a_{\overline{i + 1}}) = \mu(a_{\overline{i + 1}}, a_{\bar{i}}) = a_{\overline{i + 2}},\]
others are all zero.

Then for any linear map $\alpha \in  End(V)$, $(V, \mu, \alpha)$ is a Hom-Jordan algebra. Next, we prove that $(V, \mu, \alpha)$ is simple. Clearly, we have $\mu(V, V) = V$.

Let $W$ be a nonzero Hom-ideal of $(V, \mu, \alpha)$, then there exists a nonzero element $x = \sum^{n - 1}_{i = 0}x_{i}a_{\bar{i}} \in W$. Suppose that $x_{t} \neq 0$. Since $x_{t}a_{\overline{t + 2}} = \mu(a_{\bar{t}}, \mu(a_{\overline{t - 1}}, x)) \in W$, we have $a_{\overline{t + 2}} \in W$, so $a_{\overline{n - 2}}, a_{\overline{n - 1}} \in W$. So $a_{\bar{0}} = \mu(a_{\overline{n - 2}}, a_{\overline{n - 1}}) \in W$, $a_{\bar{1}} = \mu(a_{\overline{n - 1}}, a_{\bar{0}}) \in W$. Hence, we have all $a_{\bar{i}}(i \in \mathbb{Z}_{n}) \in W$. Therefore $W =V$ and $(V, \mu, \alpha)$ is simple.
\end{proof}
According to Theorem \ref{thm:3.5} (1) and Corollary \ref{cor:2.13}, the dimension of a simple multiplicative Hom-Jordan algebra can only be an integer multiple of dimensions of simple Jordan algebras.

By Theorem \ref{thm:3.5} (1) and Corollary \ref{cor:2.13}, in order to classify simple multiplicative Hom-Jordan algebras, we just classify automorphism on their induced Jordan algebras, in particular, automorphism on semi-simple Jordan algebras which are direct sum of finite isomorphic simple ideals.
\begin{thm}\label{thm:4.2}
Let $J$ be a semi-simple Jordan algebra and its $n$ simple ideals are isomorphic mutually, moreover $J$ can be generated by its automorphism $\alpha$(or $\beta$) and any simple ideal. Taking $\alpha_{n} = \alpha^{n}$, $\beta_{n} = \beta^{n}$, then $\alpha_{n}$(or $\beta_{n}$) leaves each simple ideal of $J$ invariant. Then there exists an automorphism $\varphi$ on $J$ satisfying $\varphi \circ \alpha = \beta \circ \varphi$ if and only if there exists an automorphism $\phi$ on the simple ideal of $J$ satisfying $\phi \circ \alpha^{n} = \beta^{n} \circ \phi$.
\end{thm}
\begin{proof}
Let $J_{1}$ be a simple ideal of $J$. Since $\alpha_{n}$(or $\beta_{n}$) leaves each simple ideal of $J$ invariant, we have $\alpha^{n}(J_{1}) = J_{1}$(or $\beta^{n}(J_{1}) = J_{1}$) and we have
\[J = J_{1} \oplus \alpha(J_{1}) \oplus \cdots \oplus \alpha^{n - 1}(J_{1})\]
\[or\; J = J_{1} \oplus \beta(J_{1}) \oplus \cdots \oplus \beta^{n - 1}(J_{1})\]
since $J$ can be generated by its automorphism $\alpha$(or $\beta$) and any simple ideal.

Choose a basis $x = (x_{1}, x_{2}, \cdots, x_{m})$ of $J_{1}$, then
\[x^{'} = (x, \alpha(x), \alpha^{2}(x), \cdots, \alpha^{n - 1}(x)),\quad x^{''} = (x, \beta(x), \beta^{2}(x), \cdots, \beta^{n - 1}(x))\]
are both bases of $J$. Let $\alpha(x^{'}) = x^{'}A$, $\beta(x^{''}) = x^{''}B$, then
$$A = \begin{pmatrix}
0& & & &A_{1}\\
I&0& & & \\
 &I&0& & \\
 & &\ddots&\ddots& \\
 & & &I&0
\end{pmatrix}
,
B = \begin{pmatrix}
0& & & &B_{1}\\
I&0& & & \\
 &I&0& & \\
 & &\ddots&\ddots& \\
 & & &I&0
\end{pmatrix}
,
$$
where $\alpha_{n}(x) = xA_{1}$, $\beta_{n}(x) = xB_{1}$.

If there exists an automorphism $\phi$ on $J_{1}$ such that $\phi \circ \alpha_{n} = \beta_{n} \circ \phi$, let $\phi(x) = xM$, then $MA_{1} = B_{1}M$. Defining $\varphi(x^{'}) = x^{''}diag(M, \cdots, M)$. Then we have $\varphi(x^{'}) = (\phi(x), \beta(\phi(x)), \cdots, \beta^{n - 1}(\phi(x)))$. It's easy to verify that $\varphi$ is an automorphism since $\phi$ is an automorphism. Moreover,
$$\varphi \circ \alpha(x^{'}) = x^{''}
\begin{pmatrix}
M& & & & \\
 &M& & & \\
 & &\ddots& & \\
 & & &\ddots& \\
 & & & &M
\end{pmatrix}
\begin{pmatrix}
0& & & &A_{1}\\
I&0& & & \\
 &I&0& & \\
 & &\ddots&\ddots& \\
 & & &I&0
\end{pmatrix}
$$
$
= x^{''}\begin{pmatrix}
0&\cdots&\cdots&0&MA_{1}\\
M&0& & & \\
 &M&0& & \\
 & & &\ddots& \\
 & & &M&0
\end{pmatrix}
,
$
$$\beta \circ \varphi(x^{'}) = x^{''}
\begin{pmatrix}
0& & & &B_{1}\\
I&0& & & \\
 &I&0& & \\
 & &\ddots&\ddots& \\
 & & &I&0
\end{pmatrix}
\begin{pmatrix}
M& & & & \\
 &M& & & \\
 & &\ddots& & \\
 & & &\ddots& \\
 & & & &M
\end{pmatrix}
$$
$
= x^{''}\begin{pmatrix}
0&\cdots&\cdots&0&B_{1}M\\
M&0& & & \\
 &M&0& & \\
 & & &\ddots& \\
 & & &M&0
\end{pmatrix}
.
$

Note that $MA_{1} = B_{1}M$, we have $\varphi \circ \alpha = \beta \circ \varphi$.

Now suppose that there exists an automorphism $\varphi$ on $J$ satisfying $\varphi \circ \alpha = \beta \circ \varphi$. According to the proof of Lemma \ref{le:3.4}, there exists $0 \leq i \leq n - 1$ such that $\varphi(J_{1}) = \beta^{i}(J_{1})$. Then $\varphi \circ \alpha^{j}(J_{1}) = \beta^{i + j}(J_{1})(0 \leq i, j \leq n - 1)$. Let $\varphi(x) = \beta^{i}(x)M_{1}$, then $\varphi(x^{'}) = x^{''}M$, where
$$\begin{pmatrix}
M_{1}& & & \\
 &M_{1}& & & \\
 & & \ddots& & \\
 & & &\ddots& \\
 & & & &M_{1}
\end{pmatrix}
(i = 0),
$$
\begin{equation}
\bordermatrix{%
&1&\cdots&n - i&n - i + 1& &\cdots&n\cr
1 & & & &B_{1}M_{1}& & & \cr
2 & & & & &B_{1}M_{1}& & \cr
\vdots & & & & & &\ddots& \cr
i & & & & & & &B_{1}M_{1}\cr
i + 1 &M_{1}& & & & & & \cr
\vdots & &\ddots& & & & & \cr
n & & &M_{1}& & & & \cr
}(1 \leq i \leq n - 1).
\end{equation}
Defining $\phi(x) = xM_{1}$, then we have
$$\varphi(x^{'}) =
\left\{
\begin{aligned}
(\phi(x), \beta(\phi(x)), \cdots, \beta^{n - 1}(\phi(x))),(i = 0)\\
(\beta^{i}(\phi(x)), \cdots, \beta^{n - 1}(\phi(x)), \beta^{n}(\phi(x)), \cdots, \beta^{n + i - 1}(\phi(x))),(1 \leq i \leq n - 1)
\end{aligned}
\right.
$$
Therefore, $\phi$ is an automorphism on $J_{1}$ since $\varphi$ is an automorphism on $J$. Moreover, we have $\phi \circ \alpha^{n} = \beta^{n} \circ \phi$ since $\varphi \circ \alpha = \beta \circ \varphi$.
\end{proof}
By Theorem \ref{thm:4.2}, it is obvious that two simple multiplicative Hom-Jordan algebras $(V_{1}, \mu_{1}, \alpha)$ and $(V_{2}, \mu_{2}, \beta)$ are isomorphic if and only if automorphisms $\alpha^{n}$ and $\beta^{n}$ on two simple ideals (as simple Jordan algebras) of the corresponding induced Jordan algebras are conjugate.

Combing Corollary \ref{cor:2.13}, Theorems \ref{thm:3.5} (1) and \ref{thm:4.2}, we get the following theorem.
\begin{thm}\label{thm:4.3}
All finite dimensional simple multiplicative Hom-Jordan algebras can be denoted as $(X, n, \Gamma_{\alpha})$, where $X$ represents the type of  the simple ideal (as the simple Jordan algebra) of the corresponding induced Jordan algebras, $n$ represents numbers of simple ideals, $\Gamma_{\alpha}$ represents the sets of conjugate classes of the automorphism $\alpha^{n}$ on the simple Jordan algebra $X$, i.e., $\Gamma_{\alpha} = \{\phi \circ \alpha^{n} \circ \phi^{-1} | \phi \in \rm{Aut}\it{(X)}\}$.
\end{thm}
\begin{ex}\label{ex:4.4}
Suppose that $V$ is a $2$-dimensional vector space with basis $\{e_{1}, e_{2}\}$. Define $\mu : V \times V \rightarrow V$ to be a bilinear map by
$$
\left\{
\begin{aligned}
\mu(e_{1}, e_{1}) = e_{2},\\
\mu(e_{2}, e_{2}) = e_{1},\\
\mu(e_{1}, e_{2}) = \mu(e_{2}, e_{1}) = 0,
\end{aligned}
\right.
$$
and $\alpha : V \rightarrow V$ a linear map by
$$
\left\{
\begin{aligned}
\alpha(e_{1}) = e_{2},\\
\alpha(e_{2}) = e_{1}.
\end{aligned}
\right.
$$
Then $(V, \mu, \alpha)$ is a simple multiplicative Hom-Jordan algebra. Moreover, its induced Jordan algebra is $(V, \mu^{'})$ where $\mu^{'} : V \times V \rightarrow V$ satisfies
$$
\left\{
\begin{aligned}
\mu^{'}(e_{1}, e_{1}) = e_{1},\\
\mu^{'}(e_{2}, e_{2}) = e_{2},\\
\mu^{'}(e_{1}, e_{2}) = \mu(e_{2}, e_{1}) = 0.
\end{aligned}
\right.
$$
$(V, \mu^{'})$ is semi-simple and has the decomposition of simple ideals $V = V_{1} \oplus V_{2}$ where $V_{1}, V_{2}$ are simple ideals generated by $e_{1}, e_{2}$ respectively. Moreover, we get $V_{1}$ is isomorphic to $V_{2}$.

According to Theorem \ref{thm:4.3}, $(V, \mu, \alpha)$ can be denoted as $(V_{1}, 2, \alpha^{2})$ or $(V_{2}, 2, \alpha^{2})$.
\end{ex}
\section{Bimodules of simple multiplicative Hom-Jordan algebras}\label{se:5}
In this section, we mainly study bimodules of simple multiplicative Hom-Jordan algebras. We will give a theorem on relationships between bimodules of Jordan-type Hom-Jordan algebras and modules of their induced Jordan algebras. Moreover, some propositions about bimoudles of simple multiplicative Hom-Jordan algebras are also obtained.
\begin{defn}\cite{A1}
Let $(V, \mu, \alpha)$ be a Hom-Jordan algebra. A $V$-bimodule is a Hom-module $(W, \alpha_{W})$ that comes equipped with a left structure map $\rho_{l} : V \otimes W \rightarrow W(\rho_{l}(a \otimes w) = a \cdot w)$ and a right structure map $\rho_{r} : W \otimes V \rightarrow W(\rho_{l}(w \otimes a) = w \cdot a)$ such that the following conditions:
\begin{enumerate}[(1)]
\item $\rho_{r} \circ \tau_{1} = \rho_{l}$ where $\tau_{1} : V \otimes W \rightarrow W \otimes V$, $a \otimes w \mapsto w \otimes a$;
\item $\alpha_{W}(w \cdot a) \cdot \alpha(\mu(b, c)) + \alpha_{W}(w \cdot b) \cdot \alpha(\mu(c, a)) + \alpha_{W}(w \cdot c) \cdot \alpha(\mu(a, b)) = (\alpha_{W}(w) \cdot \mu(b, c)) \cdot \alpha^{2}(a) + (\alpha_{W}(w) \cdot \mu(c, a)) \cdot \alpha^{2}(b) + (\alpha_{W}(w) \cdot \mu(a, b)) \cdot \alpha^{2}(c)$;
\item $\alpha_{W}(w \cdot a) \cdot \alpha(\mu(b, c)) + \alpha_{W}(w \cdot b) \cdot \alpha(\mu(c, a)) + \alpha_{W}(w \cdot c) \cdot \alpha(\mu(a, b)) = ((w \cdot a) \cdot \alpha(b)) \cdot \alpha^{2}(c) + ((w \cdot c) \cdot \alpha(b)) \cdot \alpha^{2}(a) + \mu(\mu(a, c), \alpha(b)) \cdot \alpha_{W}^{2}(w)$;
\end{enumerate}
hold for all $a, b, c \in V$, $w \in W$.
\end{defn}
\begin{ex}
In \cite{A1}, we know that $(V, \alpha)$ is a $V$-bimodule where the structure maps are $\rho_{l} = \rho_{r} = \mu$ where $(V, \mu, \alpha)$ is a Hom-Jordan algebra.
\end{ex}
Next, we construct an example which come true on a field of prime characteristic.
\begin{ex}
Let $(V, \mu, \alpha)$ be the Hom-Jordan algebra in Example \ref{ex:4.4} and $W$ a $1$-dimensional vector space with basis $\{w_{1}\}$. Define $\alpha_{W} : W \rightarrow W$ to be a linear map by $\alpha_{W}(w_{1}) = w_{1}$. Define $\rho_{l} : V \otimes W \rightarrow W$ to be a linear map by
    $$
    \left\{
    \begin{aligned}
    \rho_{l}(e_{1}, w_{1}) = w_{1},\\
    \rho_{l}(e_{2}, w_{1}) = w_{1},
    \end{aligned}
    \right.
    $$
    and $\rho_{r} : W \otimes V \rightarrow W$ a linear map by
    $$
    \left\{
    \begin{aligned}
    \rho_{l}(w_{1}, e_{1}) = w_{1},\\
    \rho_{l}(w_{1}, e_{2}) = w_{1}.
    \end{aligned}
    \right.
    $$
    Then $(W, \alpha_{W})$ is a $V$-bimodule with the structure maps defined as above when the characteristic of the ground field is two.
\end{ex}
\begin{defn}\cite{N1}
A Jordan module is a system consisting of a vector space $V$, a Jordan algebra $J$, and two compositions $x \cdot a$, $a \cdot x$ for $x$ in $V$, $a$ in $J$ which are bilinear and satisfy
\begin{enumerate} [(i)]
\item $a \cdot x = x \cdot a$,
\item $(x \cdot a) \cdot (b \circ c) + (x \cdot b) \cdot (c \circ a) + (x \cdot c) \cdot (a \circ b) = (x \cdot (b \circ c)) \cdot a + (x \cdot (c \circ a)) \cdot b + (x \cdot (a \circ b)) \cdot c$,
\item $(((x \cdot a) \cdot b) \cdot c) + (((x \cdot c) \cdot b) \cdot a) + x \cdot (a \circ c \circ b) = (x \cdot a) \cdot (b \circ c) + (x \cdot b) \cdot (c \circ a) + (x \cdot c) \cdot (a \circ b)$,
\end{enumerate}
where $x \in V, a , b , c \in J$, and $a_{1} \circ a_{2} \circ a_{3}$ stands for $((a_{1} \circ a_{2}) \circ a_{3})$.
\end{defn}
\begin{thm}\label{thm:5.5}
Let $(V, \mu, \alpha)$ be a Jordan-type Hom-Jordan algebra with $(V, \mu^{'})$ the induced Jordan algebra.
\begin{enumerate}[(1)]
\item Let $(W, \alpha_{W})$ be a $V$-bimodule of $(V, \mu, \alpha)$ with $\rho_{l}(\rho_{r})$ the left structure map(respectively, the right structure map). Suppose that $\alpha_{W}$ is invertible and satisfies $\alpha_{W}(w \cdot a) = \alpha_{W}(w) \cdot \alpha(a)$ for all $a \in V$, $w \in W$. Then $W$ is a module of the induced Jordan algebra $(V, \mu^{'})$ with two compositions $w \cdot^{'} a = \alpha_{W}^{-1}(w \cdot a)$, $a \cdot^{'} w = \alpha_{W}^{-1}(a \cdot w)$ for all $a \in V$, $w \in W$.
\item Let $W$ be a module of the induced Jordan algebra $(V, \mu^{'})$ with two compositions $w \cdot^{'} a$, $a \cdot^{'} w$ for all $a \in V$, $w \in W$. If there exists $\alpha_{W} \in End(W)$ such that $\alpha_{W}(w \cdot^{'} a) = \alpha_{W}(w) \cdot^{'} \alpha(a)$ for all $a \in V$, $w \in W$. Then $(W, \alpha_{W})$ is a $V$-bimodule of $(V, \mu, \alpha)$ with the left structure map $\rho_{l} : V \otimes W \rightarrow W(\rho_{l}(a \otimes w) = \alpha_{W}(a \cdot^{'} w))$, the right structure map $\rho_{r} : W\otimes V \rightarrow W(\rho_{r}(w \otimes a) = \alpha_{W}(w \cdot^{'} a))$.
\end{enumerate}
\end{thm}
\begin{proof}
(1) For any $x \in W$, $a, b, c \in V$, we have
\[a \cdot^{'} x = \alpha_{W}^{-1}(a \cdot x) = \alpha_{W}^{-1}(\rho_{l}(a \otimes x)) =  \alpha_{W}^{-1}(\rho_{r} \circ \tau_{1} (a \otimes x)) = \alpha_{W}^{-1}(x \cdot a) = x \cdot^{'} a.\]
\begin{align*}
&(x \cdot^{'} a) \cdot^{'} \mu^{'}(b, c) + (x \cdot^{'} b) \cdot^{'} \mu^{'}(c, a) + (x \cdot^{'} c) \cdot^{'} \mu^{'}(a, b)\\
&= \alpha_{W}^{-1}(\alpha_{W}^{-1}(x \cdot a) \cdot \alpha^{-1}(\mu(b, c))) + \alpha_{W}^{-1}(\alpha_{W}^{-1}(x \cdot b) \cdot \alpha^{-1}(\mu(c, a)))\\
&+ \alpha_{W}^{-1}(\alpha_{W}^{-1}(x \cdot c) \cdot \alpha^{-1}(\mu(a, b)))\\
&= \alpha_{W}^{-3}(\alpha_{W}^{2}(\alpha_{W}^{-1}(x \cdot a) \cdot \alpha^{-1}(\mu(b, c))) + \alpha_{W}^{2}(\alpha_{W}^{-1}(x \cdot b) \cdot \alpha^{-1}(\mu(c, a)))\\
&+ \alpha_{W}^{2}(\alpha_{W}^{-1}(x \cdot c) \cdot \alpha^{-1}(\mu(a, b))))\\
&= \alpha_{W}^{-3}(\alpha_{W}((x \cdot a) \cdot \mu(b, c)) + \alpha_{W}((x \cdot b) \cdot \mu(c, a)) + \alpha_{W}((x \cdot c) \cdot \mu(a, b)))\\
&= \alpha_{W}^{-3}(\alpha_{W}(x \cdot a) \cdot \alpha(\mu(b, c)) + \alpha_{W}(x \cdot b) \cdot \alpha(\mu(c, a)) + \alpha_{W}(x \cdot c) \cdot \alpha(\mu(a, b)))\\
&= \alpha_{W}^{-3}((\alpha_{W}(x) \cdot \mu(b, c)) \cdot \alpha^{2}(a) + (\alpha_{W}(x) \cdot \mu(c, a)) \cdot \alpha^{2}(b) + (\alpha_{W}(x) \cdot \mu(a, b)) \cdot \alpha^{2}(c))\\
&= \alpha_{W}^{-3}(\alpha_{W}(x \cdot \alpha^{-1}(\mu(b, c))) \cdot \alpha^{2}(a) + \alpha_{W}(x \cdot \alpha^{-1}(\mu(c, a))) \cdot \alpha^{2}(b)\\
&+ \alpha_{W}(x \cdot \alpha^{-1}(\mu(a, b))) \cdot \alpha^{2}(c))\\
&= \alpha_{W}^{-3}(\alpha_{W}^{2}(x \cdot^{'} \mu^{'}(b, c)) \cdot \alpha^{2}(a) + \alpha_{W}^{2}(x \cdot^{'} \mu^{'}(c, a)) \cdot \alpha^{2}(b) + \alpha_{W}^{2}(x \cdot^{'} \mu^{'}(a, b)) \cdot \alpha^{2}(c))\\
&= \alpha_{W}^{-3}(\alpha_{W}^{2}((x \cdot^{'} \mu^{'}(b, c)) \cdot a) + \alpha_{W}^{2}((x \cdot^{'} \mu^{'}(c, a)) \cdot b) + \alpha_{W}^{2}((x \cdot^{'} \mu^{'}(a, b)) \cdot c))\\
&= \alpha_{W}^{-1}((x \cdot^{'} \mu^{'}(b, c)) \cdot a) + \alpha_{W}^{-1}((x \cdot^{'} \mu^{'}(c, a)) \cdot b) + \alpha_{W}^{-1}((x \cdot^{'} \mu^{'}(a, b)) \cdot c)\\
&= (x \cdot^{'} \mu^{'}(b, c)) \cdot^{'} a + (x \cdot^{'} \mu^{'}(c, a)) \cdot^{'} b + (x \cdot^{'} \mu^{'}(a, b)) \cdot^{'} c.
\end{align*}
\begin{align*}
&(x \cdot^{'} a) \cdot^{'} \mu^{'}(b, c) + (x \cdot^{'} b) \cdot^{'} \mu^{'}(c, a) + (x \cdot^{'} c) \cdot^{'} \mu^{'}(a, b)\\
&= \alpha_{W}^{-3}(\alpha_{W}(x \cdot a) \cdot \alpha(\mu(b, c)) + \alpha_{W}(x \cdot b) \cdot \alpha(\mu(c, a)) + \alpha_{W}(x \cdot c) \cdot \alpha(\mu(a, b)))\\
&= \alpha_{W}^{-3}(((x \cdot a) \cdot \alpha(b)) \cdot \alpha^{2}(c) + ((x \cdot c) \cdot \alpha(b)) \cdot \alpha^{2}(a) + \mu(\mu(a, c), \alpha(b)) \cdot \alpha_{W}^{2}(x))\\
&= \alpha_{W}^{-3}((\alpha_{W}(x \cdot^{'} a) \cdot \alpha(b)) \cdot \alpha^{2}(c) + (\alpha_{W}(x \cdot^{'} c) \cdot \alpha(b)) \cdot \alpha^{2}(a)\\
&+ \alpha (\mu^{'}(\alpha(\mu^{'}(a, c)), \alpha(b))) \cdot \alpha_{W}^{2}(x))\\
&= \alpha_{W}^{-3}(\alpha_{W}((x \cdot^{'} a) \cdot b) \cdot \alpha^{2}(c) + \alpha_{W}((x \cdot^{'} c) \cdot b) \cdot \alpha^{2}(a) + \alpha^{2}(\mu^{'}(\mu^{'}(a, c), b)) \cdot \alpha_{W}^{2}(x))\\
&= \alpha_{W}^{-3}(\alpha_{W}^{2}((x \cdot^{'} a) \cdot^{'} b) \cdot \alpha^{2}(c) + \alpha_{W}^{2}((x \cdot^{'} c) \cdot^{'} b) \cdot \alpha^{2}(a) + \alpha^{2}(\mu^{'}(\mu^{'}(a, c), b)) \cdot \alpha_{W}^{2}(x))\\
&= \alpha_{W}^{-3}(\alpha_{W}^{2}(((x \cdot^{'} a) \cdot^{'} b) \cdot c) + \alpha_{W}^{2}(((x \cdot^{'} c) \cdot^{'} b) \cdot a)+ \alpha_{W}^{2}(\mu^{'}(\mu^{'}(a, c), b) \cdot x))\\
&= \alpha_{W}^{-1}(((x \cdot^{'} a) \cdot^{'} b) \cdot c) + \alpha_{W}^{-1}(((x \cdot^{'} c) \cdot^{'} b) \cdot a) + \alpha_{W}^{-1}(\mu^{'}(\mu^{'}(a, c), b) \cdot x)\\
&= ((x \cdot^{'} a) \cdot^{'} b) \cdot^{'} c + ((x \cdot^{'} c) \cdot^{'} b) \cdot^{'} a + \mu^{'}(\mu^{'}(a, c), b) \cdot^{'} x.
\end{align*}
Therefore, $W$ is a module of the induced Jordan algebra $(V, \mu^{'})$.

(2) For any $w \in W$, $a, b, c \in V$, we have
\[\rho_{r} \circ \tau_{1}(a \otimes w) = \alpha_{W}(w \cdot^{'} a) = \alpha_{W}(a \cdot^{'} w) = \rho_{l}(a \otimes w),\]
which implies that $\rho_{r} \circ \tau_{1} = \rho_{l}$.
\begin{align*}
&\alpha_{W}(w \cdot a) \cdot \alpha(\mu(b, c)) + \alpha_{W}(w \cdot b) \cdot \alpha(\mu(c, a)) + \alpha_{W}(w \cdot c) \cdot \alpha(\mu(a, b))\\
&= \alpha_{W}(\alpha_{W}^{2}(w \cdot^{'} a) \cdot^{'} \alpha^{2}(\mu^{'}(b, c))) + \alpha_{W}(\alpha_{W}^{2}(w \cdot^{'} b) \cdot^{'} \alpha^{2}(\mu^{'}(c, a)))\\
&+ \alpha_{W}(\alpha_{W}^{2}(w \cdot^{'} c) \cdot^{'} \alpha^{2}(\mu^{'}(a, b)))\\
&= \alpha_{W}(\alpha_{W}^{2}((w \cdot^{'} a) \cdot^{'} \mu^{'}(b, c)) + \alpha_{W}^{2}((w \cdot^{'} b) \cdot^{'} \mu^{'}(c, a)) + \alpha_{W}^{2}((w \cdot^{'} c) \cdot^{'} \mu^{'}(a, b)))\\
&= \alpha_{W}^{3}((w \cdot^{'} a) \cdot^{'} \mu^{'}(b, c) + (w \cdot^{'} b) \cdot^{'} \mu^{'}(c, a) + (w \cdot^{'} c) \cdot^{'} \mu^{'}(a, b))\\
&= \alpha_{W}^{3}((w \cdot^{'} \mu^{'}(b, c)) \cdot^{'} a + (w \cdot^{'} \mu^{'}(c, a)) \cdot^{'} b + (w \cdot^{'} \mu^{'}(a, b)) \cdot^{'} c)\\
&= \alpha_{W}^{2}(\alpha_{W}(w \cdot^{'} \mu^{'}(b, c)) \cdot^{'} \alpha(a)) + \alpha_{W}^{2}(\alpha_{W}(w \cdot^{'} \mu^{'}(c, a)) \cdot^{'} \alpha(b))\\
&+ \alpha_{W}^{2}(\alpha_{W}(w \cdot^{'} \mu^{'}(a, b)) \cdot^{'} \alpha(c))\\
&= \alpha_{W}(\alpha_{W}^{2}(w \cdot^{'} \mu^{'}(b, c)) \cdot^{'} \alpha^{2}(a)) + \alpha_{W}(\alpha_{W}^{2}(w \cdot^{'} \mu^{'}(c, a)) \cdot^{'} \alpha^{2}(b))\\
&+ \alpha_{W}(\alpha_{W}^{2}(w \cdot^{'} \mu^{'}(a, b)) \cdot^{'} \alpha^{2}(c))\\
&= \alpha_{W}^{2}(w \cdot^{'} \mu^{'}(b, c)) \cdot \alpha^{2}(a) + \alpha_{W}^{2}(w \cdot^{'} \mu^{'}(c, a)) \cdot \alpha^{2}(b) + \alpha_{W}^{2}(w \cdot^{'} \mu^{'}(a, b)) \cdot \alpha^{2}(c)\\
&= \alpha_{W}(\alpha_{W}(w) \cdot^{'} \alpha(\mu^{'}(b, c))) \cdot \alpha^{2}(a) + \alpha_{W}(\alpha_{W}(w) \cdot^{'} \alpha(\mu^{'}(c, a))) \cdot \alpha^{2}(b)\\
&+ \alpha_{W}(\alpha_{W}(w) \cdot^{'} \alpha(\mu^{'}(a, b))) \cdot \alpha^{2}(c)\\
&= (\alpha_{W}(w) \cdot \mu(b, c)) \cdot \alpha^{2}(a) + (\alpha_{W}(w) \cdot \mu(c, a)) \cdot \alpha^{2}(b) + (\alpha_{W}(w) \cdot \mu(a, b)) \cdot \alpha^{2}(c).
\end{align*}
\begin{align*}
&\alpha_{W}(w \cdot a) \cdot \alpha(\mu(b, c)) + \alpha_{W}(w \cdot b) \cdot \alpha(\mu(c, a)) + \alpha_{W}(w \cdot c) \cdot \alpha(\mu(a, b))\\
&= \alpha_{W}^{3}((w \cdot^{'} a) \cdot^{'} \mu^{'}(b, c) + (w \cdot^{'} b) \cdot^{'} \mu^{'}(c, a) + (w \cdot^{'} c) \cdot^{'} \mu^{'}(a, b))\\
&= \alpha_{W}^{3}(((w \cdot^{'} a) \cdot^{'} b) \cdot^{'} c + ((w \cdot^{'} c) \cdot^{'} b) \cdot^{'} a + \mu^{'}(\mu^{'}(a, c), b) \cdot^{'} w)\\
&= \alpha_{W}^{2}(\alpha_{W}((w \cdot^{'} a) \cdot^{'} b) \cdot^{'} \alpha(c)) + \alpha_{W}^{2}(\alpha_{W}((w \cdot^{'} c) \cdot^{'} b) \cdot^{'} \alpha(a))\\
&+ \alpha_{W}^{2}(\alpha(\mu^{'}(\mu^{'}(a, c), b)) \cdot^{'} \alpha_{W}(w))\\
&= \alpha_{W}(\alpha_{W}^{2}((w \cdot^{'} a) \cdot^{'} b) \cdot^{'} \alpha^{2}(c)) + \alpha_{W}(\alpha_{W}^{2}((w \cdot^{'} c) \cdot^{'} b) \cdot^{'} \alpha^{2}(a))\\
&+ \alpha_{W}(\alpha^{2}(\mu^{'}(\mu^{'}(a, c), b)) \cdot^{'} \alpha_{W}^{2}(w))\\
&= \alpha_{W}^{2}((w \cdot^{'} a) \cdot^{'} b) \cdot \alpha^{2}(c) + \alpha_{W}^{2}((w \cdot^{'} c) \cdot^{'} b) \cdot \alpha^{2}(a) + \alpha^{2}(\mu^{'}(\mu^{'}(a, c), b)) \cdot \alpha_{W}^{2}(w)\\
&= \alpha_{W}(\alpha_{W}(w \cdot^{'} a) \cdot^{'} \alpha(b)) \cdot \alpha^{2}(c) + \alpha_{W}(\alpha_{W}(w \cdot^{'} c) \cdot^{'} \alpha(b)) \cdot \alpha^{2}(a)\\
&+ \alpha(\mu^{'}(\alpha(\mu^{'}(a, c)), \alpha(b))) \cdot \alpha_{W}^{2}(w)\\
&= ((w \cdot a) \cdot \alpha(b)) \cdot \alpha^{2}(c) + ((w \cdot c) \cdot \alpha(b)) \cdot \alpha^{2}(a) + \mu(\mu(a, c), \alpha(b)) \cdot \alpha_{W}^{2}(w).
\end{align*}
Therefore, $(W, \alpha_{W})$ is a $V$-bimodule of $(V, \mu, \alpha)$.
\end{proof}
\begin{defn}
For a bimodule $(W, \alpha_{W})$ of a Hom-Jordan algebra $(V, \mu, \alpha)$, if a subspace $W_{0} \subseteq W$ satisfies that $\rho_{l}(a \otimes w) \in W_{0}$ for any $a \in V$, $w \in W_{0}$ and $\alpha_{W}(W_{0}) \subseteq W_{0}$, then $(W_{0}, \alpha_{W}|_{W_{0}})$ is called a $V$-submodule of $(W, \alpha_{W})$.

A bimodule $(W, \alpha_{W})$ of a Hom-Jordan algebra $(V, \mu, \alpha)$ is called irreducible, if it has precisely two $V$-submodules (itself and $0$) and it is called completely reducible if $W = W_{1} \oplus W_{2} \oplus \cdots \oplus W_{s}$, where $(W_{i}, \alpha_{W}|_{W_{i}})$ are irreducible $V$-submodules.
\end{defn}
\begin{prop}
Suppose that $(W, \alpha_{W})$ is a bimodule of simple multiplicative Hom-Jordan algebra $(V, \mu, \alpha)$ with $\alpha_{W}(a \cdot w) = \alpha(a) \cdot \alpha_{W}(w)$ for all $a \in V$, $w \in W$. Then, $Ker(\alpha_{W})$, $Im(\alpha_{W})$ are submodules of $W$ for $(V, \mu, \alpha)$. Moreover, we have an isomorphism of $(V, \mu, \alpha)$-modules $\overline{\alpha_{W}} : W/Ker(\alpha_{W}) \rightarrow Im(\alpha_{W})$.
\end{prop}
\begin{proof}
For any $w \in Ker(\alpha_{W})$, we have
\[\alpha_{W}(a \cdot w) = \alpha(a) \cdot \alpha_{W}(w) = 0,\quad\forall a \in V,\]
which implies that $a \cdot w \in Ker(\alpha_{W})$. Obviously, $\alpha_{W}(Ker(\alpha_{W})) \subseteq Ker(\alpha_{W})$. Therefore, $Ker(\alpha_{W})$ is a submodule of $W$ for $(V, \mu, \alpha)$.

For any $w \in Im(\alpha_{W})$, $a \in V$, there exists $u \in W$, $\tilde{a} \in V$ such that $w = \alpha_{W}(u)$, $a = \alpha(\tilde{a})$. Then
\[a \cdot w = \alpha(\tilde{a}) \cdot \alpha_{W}(u) = \alpha_{W}(\tilde{a} \cdot u) \in Im(\alpha_{W}),\]
it's obvious that $\alpha_{W}(Im(\alpha_{W})) \subseteq Im(\alpha_{W})$. So $Im(\alpha_{W})$ is a submodule of $W$ for $(V, \mu, \alpha)$.

Define $\overline{\alpha_{W}} : W/Ker(\alpha_{W}) \rightarrow Im(\alpha_{W})$ by $\overline{\alpha_{W}}(\bar{w}) = \alpha_{W}(w)$. It's easy to verify that $\overline{\alpha_{W}}$ is an isomorphism.
\end{proof}
\begin{cor}
If $(W, \alpha_{W})$ is a irreducible bimodule of simple multiplicative Hom-Jordan algebra $(V, \mu, \alpha)$ with $\alpha_{W}(a \cdot w) = \alpha(a) \cdot \alpha_{W}(w)$ for all $a \in V$, $w \in W$. Then $\alpha_{W}$ is invertible.
\end{cor}
\begin{prop}
Suppose that $(V, \mu, \alpha)$ is a simple multiplicative Hom-Jordan algebra and $(W, \alpha_{W})$ is a bimodule with $\alpha_{W}(a \cdot w) = \alpha(a) \cdot \alpha_{W}(w)$ for all $a \in V$, $w \in W$. Moreover, $\alpha_{W}$ is invertible. If $W$ is an irreducible module of the induced Jordan algebra $(V, \mu^{'})$ with two compositions $w \cdot^{'} a = \alpha_{W}^{-1}(w \cdot a)$, $a \cdot^{'} w = \alpha_{W}^{-1}(a \cdot w)$ for all $a \in V$, $w \in W$, then $(W, \alpha_{W})$ is an irreducible bimodule of $(V, \mu, \alpha)$.
\end{prop}
\begin{proof}
Assume that $(W, \alpha_{W})$ is reducible. Then there exists $W_{0} \neq \{0_{W}\}$ a subspace of $W$ such that $(W_{0}, \alpha_{W}|_{W{0}})$ is a submodule of $(W, \alpha_{W})$. That is $\alpha_{W}(W_{0}) \subseteq W_{0}$ and $a \cdot w \in W_{0}$, for any $a \in V$, $w \in W_{0}$. Hence, $a \cdot^{'} w = \alpha_{W}^{-1}(a \cdot w) \in \alpha_{W}^{-1}(W_{0}) = W_{0}$. So $W_{0}$ is a non trivial submodule of $W$ for $(V, \mu^{'})$, contradiction. Hence, $(W, \alpha_{W})$ is an irreducible bimodule of $(V, \mu, \alpha)$.
\end{proof}
\begin{re}
In \cite{Y1}, the author introduced another definition of Hom-Jordan algebras, depending on which one could verify that  all above results  are also valid.
\end{re}

{\bf ACKNOWLEDGEMENTS}\quad The authors would like to thank the referee for valuable comments and suggestions on this article.

\end{document}